\documentclass[a4paper]{amsart}

\usepackage{changebar}

\usepackage{microtype}
\usepackage{amscd}
\usepackage{amsmath,amssymb,amsfonts, accents}
\usepackage{xstring}
\usepackage{xcolor}
\usepackage[mathscr]{euscript}
\usepackage{xypic}
\usepackage[cmtip,all,2cell]{xy}
\entrymodifiers={+!!<0pt,\fontdimen22\textfont2>}
\usepackage{url}
\usepackage{latexsym}
\usepackage{amsthm}
\usepackage[latin1]{inputenc}
\usepackage{multicol}
\usepackage{enumitem}
\usepackage[colorlinks=true, linkcolor={blue!20!black}, pdfhighlight=/O, ocgcolorlinks=true]{hyperref}
\usepackage{graphicx}
\usepackage{color}

\numberwithin{equation}{section}
\theoremstyle{plain}
\newtheorem{theorem}[equation]{Theorem}

\newtheorem{lemma}[equation]{Lemma}

\newtheorem{proposition}[equation]{Proposition}

\newtheorem{corollary}[equation]{Corollary}

\theoremstyle{definition}

\newtheorem{definition}[equation]{Definition}
\newtheorem{variant}[equation]{Variant}

\newtheorem{example}[equation]{Example}

\newtheorem{remark}[equation]{Remark}

\newtheoremstyle{TheoremNum} 
        {\topsep}{\topsep}              
        {\itshape}                      
        {}                              
        {\bfseries}                     
        {.}                             
        { }                             
        {\thmname{#1}\thmnote{ \bfseries #3}}
    \theoremstyle{TheoremNum}

{ \begin{itemize}
    \setlength{\itemsep}{0pt}
    \setlength{\parskip}{0pt}
    \setlength{\parsep}{0pt}     }
{ \end{itemize}                  } 
{ \begin{enumerate}
    \setlength{\itemsep}{0pt}
    \setlength{\parskip}{0pt}
    \setlength{\parsep}{0pt}     }
{ \end{enumerate}                  }

\newcommand{\field}[1]{\mathbb{#1}}

\newcommand{\ol}[1]{\overline{#1}}
\newcommand{\cat}[1]{
\StrLen{#1}[\mystrlen]
\ifnum\mystrlen=1 \mathbb{#1}
\else \mathrm{#1}
\fi}

\newcommand{\ope}[1]{
\StrLen{#1}[\mystrlengt]
\ifnum\mystrlengt=1 \mathscr{#1}
\else \mathrm{#1}
\fi
}

\newcommand{\Set}[0]{\cat{Set}}
\newcommand{\sSet}[0]{\cat{sSet}}

\newcommand{\colim}{\operatornamewithlimits{\mathrm{colim}}}
\newcommand{\hocolim}{\operatornamewithlimits{\mathrm{hocolim}}}

\newcommand{\dau}[0]{\partial}

\newcommand{\mb}[1]{\mathbf{#1}}
\newcommand{\ho}[0]{\mathrm{Ho}}
\newcommand{\mm}[1]{\mathrm{#1}}

\newcommand{\Hom}[0]{\mm{Hom}}

\newcommand{\rt}[0]{\rightarrow}
\newcommand{\lt}[0]{\leftarrow}

\newcommand{\Del}[0]{\mb{\Delta}}
\newcommand{\op}[0]{\mm{op}}

\DeclareMathSymbol{\widetildesym}{\mathord}{largesymbols}{"65}

\DeclareMathSymbol{\widehatsym}{\mathord}{largesymbols}{"62}

\colorlet{RED}{red}

\title{An extension of Quillen's Theorem B}
\author{Ieke Moerdijk}
\email{i.moerdijk@uu.nl \and j.j.nuiten@uu.nl}
\author{Joost Nuiten}
\address{Mathematical Institute\\ Utrecht University\\ P.O. Box 80010\\ 3508 TA Utrecht\\ The
Netherlands.}

\begin{document}
\begin{abstract}
We prove a general version of Quillen's Theorem B, for actions of
simplicial categories, in an arbitrary left Bousfield localization of
the homotopy theory of simplicial presheaves over a site.
As special cases, we recover a version of the group completion theorem
in this general context, as well a version of Puppe's theorem on the
stability of homotopy colimits in an $\infty$-topos, due to Rezk.
\end{abstract}
\maketitle


\section{Introduction}
Theorem B is one of the first results in Quillen's influential paper `Higher K-Theory I' \cite{qui73} and as such plays an important role in the foundations of algebraic K-theory. For a functor $f\colon \cat{D}\rt \cat{C}$ between small categories, this theorem provides a way to identify the homotopy fibre of the induced map $B\cat{D}\rt B\cat{C}$ between classifying spaces: it is the classifying space of the over-category $\cat{D}/x$, provided that for each morphism $x\rt y$ in $\cat{C}$, the functor $\cat{D}/x \rt \cat{D}/y$ induces a weak equivalence between the associated classifying spaces. This condition can also be phrased by saying that the classifying spaces of these various $\cat{D}/x$ form a diagram of spaces over $\cat{C}$, on which $\cat{C}$ acts by weak equivalences. From this point of view, the theorem is very close to other results in the literature, such as Volker Puppe's theorem \cite{pup74} on homotopy colimits of homotopy cartesian diagrams. A version of this theorem also holds for actions by homology equivalences, and this version yields the group completion theorem \cite{mcd75,qui71,moe89} and Bott periodicity \cite{har80}.

These results all predate the development of Quillen model categories and their left Bousfield localizations, the homotopy theory of simplicial presheaves and sheaves, and the theory of $\infty$-categories and $\infty$-toposes. The purpose of this paper is to reconsider Quillen's Theorem B in the light of these developments. We will prove a very general version of Theorem B over an arbitrary site, for actions of a presheaf of simplicial categories on another simplicial presheaf. This general theorem states that if the action is by weak equivalences in some further left Bousfield localization of one of the standard model structures on simplicial presheaves, then the fibre and the homotopy fibre of the action become equivalent in this localization. See Theorem \ref{thm:quillenb} below for a precise formulation.

This theorem has the expected applications, such as a version of the group completion theorem for actions of presheaves of simplicial monoids (such as the classifying space of a coproduct $\coprod_n BGL_n(R)$ for a sheaf of rings $R$) (Examples \ref{ex:groupcompletion} and \ref{ex:bglr} below) and a version of Puppe's theorem for homotopy cartesian morphisms between diagrams of simplicial presheaves over a site (Example \ref{ex:puppe} below). When applied to a left exact localization of simplicial presheaves, this result gives precisely what is sometimes referred to as (Rezk) descent for $\infty$-toposes \cite{rez98}.
When $R$ is a sheaf of commutative rings on a site, the theorem shows that the associated projective space $\field{P}^\infty$ is $\field{A}^1$-homotopy equivalent to its group completion $\Omega B(\field{P}^\infty)$ (Example \ref{ex:pinfinity}). We expect our general version of Quillen's theorem to have further applications when applied to specific sites such as the Nisnevich topology for $\field{A}^1$-homotopy theory \cite{mor99}.

The plan of this short paper is as follows. In Sections \ref{sec:simpsh} and \ref{sec:bisimpsh} we review the homotopy theory of simplicial and bisimplicial presheaves and sheaves. This material is largely standard, and can be found in many sources of which we will mention the main ones. In Section \ref{sec:act} we introduce the necessary notation and terminology for actions by categories on simplicial presheaves, so as to state and prove the main theorem in Section \ref{sec:mainthm}. Our proof closely follows the strategy of \cite{moe89}. We conclude with some applications in Section \ref{sec:ex}.

\section{Simplicial presheaves and sheaves}\label{sec:simpsh}
In this section we review some basic definitions and facts about the homotopy theory of simplicial presheaves and sheaves. Almost everything in this section traces back to \cite{bro73, jar87, joy84}.

Let $(\cat{S}, J)$ be a site, i.e.\ a small category $\cat{S}$ equipped with a Grothendieck topology $J$. Let $\mm{PSh}(\cat{S})$ and $\mm{Sh}(\cat{S}, J)$ be the categories of presheaves (resp.\ sheaves) of \emph{sets} on $S$ and let
\begin{equation}\label{diag:assocsheaf}
\vcenter{\xymatrix{
i^*\colon \mm{PSh}(\cat{S})\ar@<1ex>[r] & \mm{Sh}(\cat{S}, J)\ar@<1ex>[l]\colon i_*
}}\end{equation}
be the adjoint pair given by the full embedding $i_*$ and the associated sheaf functor $i^*$. By adjointness $i_*$ preserves all limits and $i^*$ preserves all colimits, while in addition $i^*$ preserves finite limits.

A \emph{point} of the topos $\mm{Sh}(\cat{S}, J)$ (or ``of the site $(\cat{S}, J)$'') is such an adjoint pair
$$\xymatrix{
p^*\colon \mm{Sh}(\cat{S}, J)\ar@<1ex>[r] & \cat{Sets}\ar@<1ex>[l]\colon p_*
}$$
for which $p^*$ preserves finite limits (i.e.\ the pair forms a geometric morphism $p\colon \cat{Sets}\rt \mm{Sh}(\cat{S}, J)$). The topos $\mm{Sh}(\cat{S}, J)$ is said to have \emph{enough points} if the collection of functors $p^*$, for all points $p$ of $\mm{Sh}(\cat{S}, J)$, is jointly conservative (i.e.\ detects isomorphisms). Equivalently, $\mm{Sh}(\cat{S}, J)$ has enough points if there exists a topological space $X$ and a geometric morphism $f\colon \mm{Sh}(X)\rt \mm{Sh}(\cat{S}, J)$ for which $f^*$ is conservative (i.e.\ ``$f$ is surjective''). Many sites occurring in nature have enough points \cite{sga4, mak77} and in some definitions and arguments we will assume that there are enough points, in order to help develop some intuition and to connect to the classical homotopy theory of simplicial sets. However, this assumption is never essential and can be circumvented by either working with a surjective ``Boolean point'' or by using the internal logic of $\mm{Sh}(\cat{S}, J)$.

The adjoint pair \eqref{diag:assocsheaf} induces an adjoint pair
$$\vcenter{\xymatrix{
i^*\colon \mm{sPSh}(\cat{S})\ar@<1ex>[r] & \mm{sSh}(\cat{S}, J)\ar@<1ex>[l]\colon i_*
}}$$
between the categories of \emph{simplicial presheaves} and \emph{sheaves}. The category $\mm{sPSh}(\cat{S})$ can be endowed with the \emph{projective model structure}, for which the fibrations and weak equivalences are defined levelwise: a map $Y\rt X$ of simplicial presheaves on $\cat{S}$ is a fibration or weak equivalence if for each object $S\in \cat{S}$, the map $Y(S)\rt X(S)$ is a fibration or weak equivalence in the classical Kan-Quillen model structure on simplicial sets. All model categorical notions for presheaves will refer to the projective model structure, unless stated otherwise.

The category $\mm{sSh}(\cat{S}, J)$ carries the \emph{Joyal} or \emph{injective} model structure, for which the cofibrations are the monomorphisms and the weak equivalences are the so-called \emph{local weak equivalences}: a map $Y\rt X$ of simplicial sheaves is called a local weak equivalence iff the map $p^*Y\rt p^*X$ is a weak equivalence of simplicial sets for every point $p$. There are rather few fibrations in this model structure, but there is a wider class of so-called \emph{local fibrations}, viz.\ the maps $Y\rt X$ for which each $p^*Y\rt p^*X$ is a Kan fibration. Equivalently, these are the maps for which the map
$$\xymatrix{
Y(\Delta[n])\ar[r] & X(\Delta[n])\times_{X(\Lambda^k[n])} X(\Lambda^k[n])
}$$
is a surjection of sheaves of sets, for each $n\geq 1$ and each $0\leq k\leq n$. Here we use that each simplicial set $K$ and simplicial sheaf $Y$ determine a sheaf (of sets) $Y(K)$, determined by
\begin{align*}
Y(\Delta[n]) &= Y_n\\
Y(\colim K_i) &= \lim Y(K_i).
\end{align*}
Alternatively, using that the Joyal model structure is a simplicial model structure, one can identify $Y(K)$ with the sheaf of vertices of $Y^K$. Similarly, a \emph{local trivial fibration} is a map $Y\rt X$ for which each $p^*Y\rt p^*X$ is a trivial fibration of simplicial sets, or equivalently, for which each map
$$\xymatrix{
Y(\Delta[n])\ar[r] & X(\Delta[n])\times_{X(\dau\Delta[n])} Y(\dau\Delta[n]) & & n\geq 0
}$$
is a surjection of sheaves of sets.

One easily verifies that the adjoint pair
$$\xymatrix{
i^*\colon \mm{sPSh}(\cat{S})\ar@<1ex>[r] & \mm{sSh}(\cat{S}, J)\ar@<1ex>[l]\colon i_*
}$$
is a Quillen pair. Among general Quillen pairs, it has some additional properties that are useful to keep in mind:
\begin{itemize}
\item[(a)] $i^*$ preserves weak equivalences between arbitrary objects (not just cofibrant ones).
\item[(b)] Let us say that a map of simplicial presheaves $Y\rt X$ is a \emph{local weak equivalence} (resp.\ a \emph{local (trivial) fibration}) if its image under $i^*$ is such. Since $i^*$ preserves finite limits (as we already mentioned), it follows that any levelwise (trivial) fibration between simplicial presheaves is a local (trivial) fibration.
\end{itemize}
It follows from (a) and the fact that $i^*i_*\cong \mm{id}$ that $\field{L}i^*\field{R}i_*\simeq \mm{id}$, so that $\mm{sSh}(\cat{S}, J)$ is a \emph{localization} of $\mm{sPSh}(\cat{S})$. Since the weak equivalences in $\mm{sSh}(\cat{S}, J)$ form an accessibly embedded accessible subcategory of the arrow category of $\mm{sSh}(\cat{S}, J)$, it follows that there exists a left Bousfield localization $\mm{sPSh}(\cat{S})_J$ of (the projective model structure on) $\mm{sPSh}(\cat{S})$ whose weak equivalences are the local weak equivalences. In this way one obtains a diagram of left Quillen functors
$$\xymatrix{
\mm{sPSh}(\cat{S})\ar[d]_{\mm{id}}\ar[r]^{i^*} & \mm{sSh}(\cat{S}, J)\\
\mm{sPSh}(\cat{S})_J\ar@{..>}[ru]_{j^*}
}$$
As ordinary functors, $j^*$ and its right adjoint $j_*$ can be identified with $i^*$ and $i_*$. The pair $j^*$ and $j_*$ forms a Quillen equivalence because $\field{L}j^*\field{R}j_*\simeq \mm{id}$ and $j^*$ preserves and detects weak equivalences.

We will use the following simple observations.
\begin{lemma}\label{lem:pullbackalonglocalfib}
In $\mm{sSh}(\cat{S}, J)$, as well as in $\mm{sPSh}(\cat{S})_J$, the pullback along a local fibration is a homotopy pullback.
\end{lemma}
\begin{proof}
The two cases are proved in the same way. Let
$$\xymatrix{
W\ar[r]\ar[d] & Y\ar[d]^f\\
Z\ar[r]_g & X
}$$
be a pullback in $\mm{sSh}(\cat{S}, J)$ (or in $\mm{sPSh}(\cat{S})_J$) in which $f$ is a local fibration. The image of this pullback square under a point of $\mm{Sh}(\cat{S}, J)$ is a homotopy pullback of simplicial sets, since the usual model structure on simplicial sets is right proper. In particular, pullbacks along local fibrations preserve local weak equivalences, so that $\mm{sSh}(\cat{S}, J)$ is right proper as well.

Now let $Z\stackrel{\sim}{\hookrightarrow} Z'\rightarrow X$ be a factorization of $g$ into a local weak equivalence, followed by a fibration. Then the pullback $Z'\times_X Y$ computes the homotopy pullback of $f$ and $g$ and the map $W\rt Z'\times_X Y$ is a local weak equivalence.
\end{proof}
\begin{lemma}\label{lem:sheafificationlex}
Let $Y\rt X\lt Z$ be a diagram in $\mm{sSh}(\cat{S}, J)$. Then its homotopy pullback can be computed as $i^*(Q)$, where
$$\xymatrix{
Q\ar[r]\ar[d] & i_*(Z)\ar[d]\\
i_*(Y)\ar[r] & i_*(X)
}$$
is a homotopy pullback in $\mm{sPSh}(\cat{S})$.
\end{lemma}
\begin{proof}
Let $i_*(Y)\rt P\rt i_*(X)$ be a factorization into a weak equivalence, followed by a fibration of simplicial presheaves. Then $Q=P\times_{i_*(X)} i_*(Z)$ is the homotopy pullback in $\mm{sPSh}(\cat{S})$ since this model category is right proper. So $i^*(Q)\cong i^*(P)\times_X Z$ and $i^*(P)$ fits into a sequence $$\xymatrix{
Y\cong i^*i_*(Y)\ar[r] & i^*(P)\ar[r] & i^*i_*(X)
}$$
of a local weak equivalence, followed by a local fibration. The lemma now follows from Lemma \ref{lem:pullbackalonglocalfib}.  
\end{proof}

\section{Bisimplicial presheaves and sheaves}\label{sec:bisimpsh}
We will write $\mm{bisSh}(\cat{S}, J)$ and $\mm{bisPSh}(\cat{S})$ for the categories of bisimplicial sheaves and presheaves on the site $(\cat{S}, J)$. These carry several model structures, but we will mostly be interested in the ``diagonal'' one, making the model categories Quillen equivalent to $\mm{sPSh}(\cat{S})$ and $\mm{sSh}(\cat{S}, J)$, respectively. More precisely, write
$$\xymatrix{
i^*\colon \mm{bisPSh}(\cat{S})\ar@<1ex>[r] & \mm{bisSh}(\cat{S}, J)\colon i_*\ar@<1ex>[l]
}$$
for the associated sheaf functor $i_*$ and its fully faithful right adjoint, and let
$$\xymatrix{
\delta^*\colon \mm{bisPSh}(\cat{S})\ar[r] & \mm{sPSh}(\cat{S})
}$$
be the diagonal functor. The functor $\delta^*$ has a left adjoint $\delta_!$ and a right adjoint $\delta^*$. Using the same notation for sheaves, we obtain a diagram of adjoint pairs
$$\xymatrix{
\mm{sPSh}(\cat{S})\ar@<1ex>[r]^{i^*}\ar@<-2ex>[d]_{\delta_!}\ar@<2ex>[d]^{\delta_*}\ar@{<-}[d]|{\delta^*} & \mm{sSh}(\cat{S}, J)\ar@<1ex>[l]^{i_*}\ar@<-2ex>[d]_-{\delta_!}\ar@<2ex>[d]^-{\delta_*}\\
\mm{bisPSh}(\cat{S})\ar@<1ex>[r]^{i^*} & \mm{bisSh}(\cat{S}, J)\ar@<1ex>[l]^{i_*}\ar[u]|-{\delta^*}
}$$
which are related by the following natural isomorphisms
$$
\delta^*i^*=i^*\delta^* \qquad\qquad \delta^*i_*=i_*\delta^*
$$
and hence $\delta_!i^*=i^*\delta_!$.

\begin{proposition}[{cf.\ \cite{moe89}}]\label{prop:diagonaltransfer}
The (projective, resp.\ Joyal) model structures can be transferred along the adjoint pair $(\delta_!, \delta^*)$ and give model structures and Quillen equivalences
$$\xymatrix@R=0.5pc{
\delta_!\colon \mm{sPSh}(\cat{S})\ar@<1ex>[r] & \mm{bisPSh}(\cat{S})\ar@<1ex>[l]\colon\delta^*\\
\delta_!\colon \mm{sPSh}(\cat{S})_J\ar@<1ex>[r] & \mm{bisPSh}(\cat{S})_J\ar@<1ex>[l]\colon\delta^*\\
\delta_!\colon \mm{sSh}(\cat{S}, J)\ar@<1ex>[r] & \mm{bisSh}(\cat{S}, J)\ar@<1ex>[l]\colon\delta^*
}$$
\end{proposition}
\begin{proof}
We prove the second case; the other two cases are similar. To show that the transferred model structure exists, it suffices to verify that $\delta^*\delta_!$ maps generating trivial cofibrations to local weak equivalences that are monic. Indeed, these maps are stable under pushout and transfinite composition while $\delta^*$ and $\delta_!$ both commute with colimits. 

It is easy to check that $\delta^*\delta_!$ preserves monomorphisms. The fact that it preserves local weak equivalences follows immediately from the fact that the unit map $X\rt \delta^*\delta_!(X)$ is a levelwise weak equivalence of simplicial presheaves. Indeed, this just follows from the analogous statement for simplicial \emph{sets}: by a standard skeletal induction it suffices to verify that for every simplex $\Delta[n]$, the map $\Delta[n]\rt \delta^*\delta_!(\Delta[n])$ is a weak equivalence. But this map can be identified with the diagonal map $\Delta[n]\rt \Delta[n]\times\Delta[n]$. 

Similarly, the fact that $X\rt \delta^*\delta_!(X)$ is a levelwise weak equivalence shows that the Quillen pair is a Quillen equivalence (because $\delta^*$ preserves and detects weak equivalences).
\end{proof}
\begin{remark}
Since $\delta^*$ preserves monomorphisms and weak equivalences, the pair $\delta^*\colon \mm{bisSh}(\cat{S}, J)\leftrightarrows \mm{sSh}(\cat{S}, J)\colon \delta_*$ is a Quillen pair as well. 
\end{remark}
The proof of Proposition \ref{prop:diagonaltransfer} applies equally well to further left Bousfield localizations of these model categories. More precisely, let $\lambda$ be a set of maps (which one can always take to be cofibrations) in $\mm{sPSh}(\cat{S})$ and let us denote by
$$\xymatrix{
i^*\colon\mm{sPSh}(\cat{S})_{J, \lambda}\ar@<1ex>[r] & \mm{sSh}(\cat{S}, J)_\lambda\ar@<1ex>[l]\colon i_*\\
}$$
the associated Quillen equivalence between the left Bousfield localizations at $\lambda$ and $i^*(\lambda)$, respectively. We will refer to the weak equivalences in these model structures as \emph{$\lambda$-equivalences} (leaving the reference to $J$ implicit when working with simplicial presheaves). The argument of Proposition \ref{prop:diagonaltransfer} shows that these two model structures can be transferred to model structures on bisimplicial (pre)sheaves along $(\delta_!, \delta^*)$, yielding two Quillen equivalences
$$\xymatrix@R=0.5pc{
\delta_!\colon \mm{sPSh}(\cat{S})_{J, \lambda}\ar@<1ex>[r] & \mm{bisPSh}(\cat{S})_{J, \lambda}\ar@<1ex>[l]\colon \delta^*\\
\delta_!\colon \mm{sPSh}(\cat{S}, J)_{\lambda}\ar@<1ex>[r] & \mm{bisSh}(\cat{S}, J)_{\lambda}\ar@<1ex>[l]\colon \delta^*.
}$$
In fact, the transferred model structure $\mm{bisPSh}(\cat{S})_{J, \lambda}$ is simply the left Bousfield localization of $\mm{bisPSh}(\cat{S})_{J}$ at the set of maps $\delta_!(\lambda)$, and similarly for sheaves.
\begin{lemma}\label{lem:realizationlemma}
Let $f\colon X\rt Y$ be a map of bisimplicial (pre)sheaves over $\cat{S}$. If $f$ induces a $\lambda$-equivalence $X_{n, -}\rt Y_{n, -}$ of simplicial (pre)sheaves for each $n\geq 0$, then the diagonal $\delta^*X\rt \delta^*Y$ is a $\lambda$-equivalence as well.
\end{lemma}
\begin{proof}
This follows from the fact that $\delta^*\colon \mm{bisSh}(\cat{S}, J)\rt \mm{sSh}(\cat{S}, J)_\lambda$ is a left Quillen functor for the Reedy model structure on $\mm{bisSh}(\cat{S}, J)=\mm{sSh}(\cat{S}, J)^{\Del^\op}$.
\end{proof}

\section{Actions on simplicial presheaves and sheaves}\label{sec:act}
We begin with some terminology and notation. Let $\cat{C}$ be a \emph{category object} in one of the (model) categories $\mm{sPSh}(\cat{S})$ or $\mm{sSh}(\cat{S}, J)$. Thus $\cat{C}$ is given by simplicial (pre)sheaves $\mm{ob}(\cat{C})$ and $\mm{mor}(\cat{C})$ of objects and morphisms, together with structure maps for source and target
$$\xymatrix{
\mm{mor}(\cat{C})\ar@<1ex>[r]^s\ar@<-1ex>[r]_t & \mm{ob}(\cat{C})
}$$
and two more structure maps for units and composition, all satisfying the usual identities. For any such category object $\cat{C}$, its nerve $N\cat{C}$ is a bisimplicial (pre)sheaf whose diagonal we denote
$$
B\cat{C}= \delta^*N\cat{C}
$$
and call the classifying (pre)sheaf or ``space'' of $\cat{C}$. Thus, $B\cat{C}$ is an object of $\mm{sPSh}(\cat{S})$ or $\mm{sSh}(\cat{S}, J)$.

A \emph{left action} of $\cat{C}$ on a simplicial presheaf $X$ is given by maps
$$\xymatrix{
\pi\colon X\ar[r] & \mm{ob}(\cat{C}) & \text{and} & \mu\colon s^*(X)=\mm{mor}(\cat{C})\times_{\mm{ob}(\cat{C})} X\ar[r] & X
}$$
satisfying the usual identities (which express that for any $S\in\cat{S}$, the components $\pi_S$ and $\mu_S$ determine a covariant simplicial functor $\cat{C}(S)\rt \cat{sSet}$, natural in $S$). The domain of the map $\mu$ is the pullback $s^*(X)$ of $\pi$ along $s$. Such an action by $\cat{C}$ on $X$ defines a new category object $X_{\cat{C}}$ in $\mm{sPSh}(\cat{S})$ (or in $\mm{sSh}(\cat{S}, J)$) with
\begin{align*}
\mm{ob}(X_{\cat{C}}) &= X\\
\mm{mor}(X_{\cat{C}}) &= s^*(X)
\end{align*}
while the new source and target map $s^*(X)\rightrightarrows X$ are the projection and the action $\mu$. For any object $S\in\cat{S}$ and any simplicial degree $n$, the category $X_{\cat{C}}(S)_n$ (in sets) can therefore be described as follows: the objects are $n$-simplices $x\in X(S)_n$ and a morphism $x\rt y$ is a morphism $\phi\colon \pi(x)\rt \pi(y)$ in the category $\cat{C}(S)_n$ such that $\mu(\phi, x)=y$. There is an obvious projection functor
$$\xymatrix{
\pi\colon X_{\cat{C}}\ar[r] & X
}$$
which induces a map of classifying spaces
$$\xymatrix{
B\pi\colon BX_{\cat{C}}\ar[r] & B\cat{C}.
}$$
For any $n$-simplex $c\in \mm{ob}(\cat{C})(S)_n$, i.e.\ a map $S\times \Delta[n]\rt \mm{ob}(\cat{C})$ of simplicial presheaves, we write $X(c)$ for the pullback
$$\xymatrix{
X(c)\ar[r]\ar[d] & X\ar[d]^\pi\\
S\times\Delta[n]\ar[r]_-c & \mm{ob}(\cat{C})
}$$
A 0-simplex $c\colon S\times\Delta[0]\rt \mm{ob}(\cat{C})$ determines a map $S\times\Delta[0]\rt B\cat{C}$ and $X(c)$ fits into a pullback of simplicial (pre)sheaves
$$\xymatrix{
X(c)\ar[r]\ar[d] & BX_{\cat{C}}\ar[d]\\
S\times\Delta[0]\ar[r]_-c & B\cat{C}.
}$$
The action $\mu$ defines a map $\ol{\mu}=(\pi_1, \mu)$ over $\mm{mor}(\cat{C})$
$$\xymatrix{
s^*(X)\ar[rr]^-{\ol{\mu}}\ar[rd] & & t^*(X)\ar[ld]\\
 & \mm{mor}(\cat{C}).
}$$
If $\phi\in\mm{mor}(\cat{C})(S)_n$ is a morphism from $c$ to $d$, i.e.\ $\phi\colon S\times\Delta[n]\rt \mm{mor}(\cat{C})$ with $s\phi=c$ and $t\phi=d$, then $\ol{\mu}$ restricts to a map of simplicial presheaves
$$\xymatrix{
\phi_*\colon X(c)\ar[r] & X(d).
}$$
Given a set of maps $\lambda$ in $\mm{sPSh}(\cat{S})_J$ or $\mm{sSh}(\cat{S}, J)$, we can require these action maps to be weak equivalences in the resulting left Bousfield localization:
\begin{definition}
Let $\cat{C}$ be a category acting on $X$ in $\mm{sPSh}(\cat{S})$, as above. Then $\cat{C}$ is said to \emph{act by $\lambda$-equivalences} if for any object $S\in\cat{S}$ and any morphism $\phi\colon S\times\Delta[n]\rt \mm{mor}(\cat{C})$ from $c=s\phi$ to $d=t\phi$, the map $\phi_*\colon X(c)\rt X(d)$ is a $\lambda$-equivalence.
\end{definition}
There are some conditions closely related to this definition. Let us call a map of simplicial presheaves over a simplicial presheaf $B$
$$\xymatrix{
X\ar[rr]\ar[rd] & & Y\ar[ld]\\
& B
}$$
a \emph{stable $\lambda$-equivalence} if for any map $A\rt B$, the pullback $A\times_B X\rt A\times_B Y$ is a $\lambda$-equivalence.
\begin{lemma}\label{lem:actingbyweonvertices}
Suppose that the map $\pi\colon X\rt \mm{ob}(\cat{C})$ is a local fibration. Then $\cat{C}$ acts by $\lambda$-equivalences iff the condition holds for $n=0$ only, i.e.\ for every vertex in $\mm{mor}(\cat{C})(S)$.
\end{lemma}
\begin{proof}
Let $\phi\colon c\rt d$ be as in the definition and for any $i=0, \dots, n$, consider the pullback
$$\xymatrix@C=1pc@R=1pc{
X(c_i)\ar[dd]\ar[rd]^-{(\phi_i)_*}\ar[rrr]^-{w_i} & & & X(c)\ar[rd]^{\phi_*}\ar[dd]|!{[lld]; [rd]}\hole\\
& X(d_i)\ar[ld]\ar[rrr]^-(0.37){u_i} & & & X(d)\ar[ld]\\
S\times\Delta[0]\ar[rrr]_-{v_i} & & & S\times\Delta[n] 
}$$
where $v_i$ is the inclusion of the $i$-th vertex and $u_i$ and $w_i$ are its pullbacks. Each of these three maps is a local weak equivalence by Lemma \ref{lem:pullbackalonglocalfib}, so that $\phi_*$ is a $\lambda$-equivalence if and only if $(\phi_i)_*$ is.
\end{proof}
\begin{lemma}\label{lem:actingbywe}
Let $\cat{C}$ be a category acting on $X$ in $\mm{sPSh}(\cat{S})$, as above. Then $\cat{C}$ acts on $X$ by $\lambda$-equivalences iff $\ol{\mu}\colon s^*(X)\rt t^*(X)$ is a stable $\lambda$-equivalence over $\mm{mor}(\cat{C})$.
\end{lemma}
\begin{proof}
Since the maps $\phi_*$ are pullbacks of $\ol{\mu}$ over $\mm{mor}(\cat{C})$, the condition of the lemma is clearly sufficient. For the converse, consider a map $A\rt \mm{mor}(\cat{C})$ and let
$$\xymatrix{
s^*(X)_A\ar[rr]\ar[rd] & & t^*(X)_A\ar[ld]\\
& A
}$$
be the pullback of $\ol{\mu}$ along $A\rt \mm{mor}(\cat{C})$. Consider the bisimplicial presheaf $\tilde{X}(s)\in \mm{bisPSh}(\cat{S})$ whose value on an object $S\in \cat{S}$ has as $(p, q)$-simplices diagrams of the form
$$\xymatrix{
\Delta[p]\ar[d]\ar[rrr] & & & s^*(X)_A(S)\ar[d]\\
\Delta[n_0]\ar[r] & \dots\ar[r] & \Delta[n_q]\ar[r] & A(S).
}$$
In the same way, let $\tilde{X}(t)$ be the bisimplicial presheaf obtained using $t^*(X)_A$ instead of $s^*(X)_A$.

For fixed $S$ and $p$, the simplicial set $\tilde{X}(s)(S)_p$ is the nerve of a category whose objects are pairs consisting of a $p$-simplex of $s^*(X)_A$ and a factorization of $\Delta[p]\rt s^*(X)_A(S)\rt A(S)$ through a simplex $\Delta[n]$ (as in the above diagram, for $q=0$). For a fixed $p$-simplex of $s^*(X)_A$, there is an initial such factorization, so that there is a (natural) weak equivalence
$$\xymatrix{
s^*(X)_{A, p}\ar[r] & \tilde{X}(s)(S)_p
}$$
from a discrete simplicial set to the simplicial set $\tilde{X}(s)(S)_p$. Taking diagonals, it follows that there is a (projective) weak equivalence of simplicial presheaves
$$\xymatrix{
s^*(X)_A\ar[r] & \delta^*\tilde{X}(s).
}$$
The same holds for $t^*(X)_A\rt \delta^*\tilde{X}(t)$, of course.

On the other hand, in each fixed simplicial degree $q$, the map $\tilde{X}(s)_{-, q}\rt \tilde{X}(t)_{-, q}$ is a coproduct of maps $\phi_*\colon X(c)\rt X(d)$, indexed by the composite maps
$$\xymatrix{
\phi\colon \Delta[n_0]\ar[r] & \dots\ar[r] & \Delta[n_q]\ar[r] & A(S)\ar[r] & \mm{mor}(\cat{C}).
}$$
These maps $\phi_*$ are $\lambda$-equivalences by assumption, so the map $\delta^*\tilde{X}(s)\rt \delta^*\tilde{X}(t)$ is a $\lambda$-equivalence as well, by Lemma \ref{lem:realizationlemma}. The commutative square
$$\xymatrix{
s^*(X)_A\ar[r]\ar[d]_\sim & t^*(X)_A\ar[d]^\sim\\
\delta^*\tilde{X}(s)\ar[r] & \delta^*\tilde{X}(t)
}$$
now shows that $s^*(X)_A\rt t^*(X)_A$ is a $\lambda$-equivalence, which finishes the proof.
\end{proof}

\section{The main theorem}\label{sec:mainthm}
In this section we will state and prove the main theorem. Some examples and applications have already been mentioned in the introduction and will be elaborated on in the next section. As before, we work over a fixed site $(\cat{S}, J)$ and consider the projective local model structure on $\mm{sPSh}(\cat{S})_J$, the injective one on $\mm{sSh}(\cat{S}, J)$, as well as left Bousfield localizations of these at a set of maps $\lambda$.
\begin{theorem}\label{thm:quillenb}
Let $\cat{C}$ be a category object acting on a simplicial presheaf $X$ by $\lambda$-equivalences. Suppose $\pi\colon X\rt \mm{ob}(\cat{C})$ is a local fibration. Then for any object $S\in\cat{S}$ and any $c\in\cat{C}(S)_0$, the map from the pullback $X(c)$ as in
\begin{equation}\label{diag:quillensquare}\vcenter{\xymatrix{
X(c)\ar[r]\ar[d] & BX_{\cat{C}}\ar[d]\\
S\times\Delta[0]\ar[r] & B\cat{C}
}}\end{equation}
to the homotopy pullback is a $\lambda$-equivalence.
\end{theorem}
\begin{remark}
Note that the theorem refers to the homotopy pullback in the projective model structure and \emph{not} in the $\lambda$-localized model structure. Of course, the two notions coincide in the case where the localization is (homotopy) left exact. This is the case where the model category $\mm{sPSh}(\cat{S})_{\lambda}$ presents an $\infty$-topos.
\end{remark}
It will be clear that our proof for presheaves applies to sheaves as well, but in fact the case of sheaves is also just a direct consequence:
\begin{corollary}
Consider a left Bousfield localization $\mm{sSh}(\cat{S}, J)_\lambda$ of the Joyal model structure. If a category object $\cat{C}$ acts on a simplicial sheaf $X$ by $\lambda$-equivalences and the map $X\rt \mm{ob}(\cat{C})$ is a local fibration, then the map 
$$\xymatrix{
X(c)\ar[r] & BX_{\cat{C}}\times^h_{B\cat{C}} \Big(S\times\Delta[0]\Big)
}$$
is a $\lambda$-equivalence, where the homotopy pullback is computed in the Joyal model structure.
\end{corollary}
\begin{proof}
Form the homotopy pullback of simplicial presheaves
$$\xymatrix{
Q\ar[r]\ar[d] & i_*\big(BX_{\cat{C}})\ar[d]\\
i_*\big(S\times \Delta[0]\big)\ar[r] & i_*(B\cat{C}).
}$$
The left Bousfield localization $\mm{sSh}(\cat{S}, J)_\lambda$ is Quillen equivalent to the left Bousfield localization $\mm{sPSh}(\cat{S})_{J, \lambda}$ and the map $i_*X(c)\rt Q$ is a $\lambda$-equivalence of simplicial presheaves. It follows that $X(c)\rt i^*Q$ is a $\lambda$-equivalence of simplicial sheaves, so that the result follows from Lemma \ref{lem:sheafificationlex}.
\end{proof}
\begin{proof}[Proof (of Theorem \ref{thm:quillenb})]
We follow the strategy from \cite{moe89}. The square \eqref{diag:quillensquare} in the theorem is obtained by applying the diagonal functor $\delta^*$ to the pullback square of bisimplicial presheaves
\begin{equation}\label{diag:quillenbsquarebeforereal}\vcenter{\xymatrix{
X(c)\ar[d]\ar[r] & N(X_{\cat{C}})\ar[d]\\
S\times\Delta[0]\ar[r] & N(\cat{C}).
}}\end{equation}
Here $X(c)$ and $S\times\Delta[0]$ are considered as bisimplicial presheaves which are constant in one simplicial direction. It thus suffices to prove the theorem for the homotopy pullback of \eqref{diag:quillenbsquarebeforereal} in $\mm{bisPSh}(\cat{S})$. This homotopy pullback can be formed by factoring the map $S\times\Delta[0]\rt N(\cat{C})$ as a trivial cofibration, followed by a fibration and then taking the pullback of $N(X_{\cat{C}})\rt N(\cat{C})$ along that fibration.

Such a factorization is obtained in the standard way from the small object argument, as a transfinite composition of pushouts of generating trivial cofibrations, i.e.\ maps $T\times \delta_!\Lambda^k[n]\rt T\times \delta_!\Delta[n]$ for any object $T\in\cat{S}$. Since pulling back along a map commutes with colimits in bisimplicial presheaves, it thus suffices to show that for any pullback diagram of the form
$$\vcenter{\xymatrix{
X_{\sigma i}\ar@{^{(}->}[r]\ar[d] & X_\sigma\ar[r]\ar[d] & N(X_{\cat{C}})\ar[d]\\
T\times \delta_!\Lambda^k[n]\ar@{^{(}->}[r]_i & T\times\delta_!\Delta[n] \ar[r] & N\cat{C}
}}$$
(where $i$ denotes the inclusion), the map $X_{\sigma i}\rt X_\sigma$ becomes a $\lambda$-equivalence after applying $\delta^*$. Indeed, then the map $\delta^*(X_{\sigma i})\rt \delta^*(X_\sigma)$ becomes a trivial cofibration in the $\lambda$-localization of the \emph{injective} model structure, and a transfinite composition of pushouts of these remains a $\lambda$-equivalence. 

Let us explicitly spell out the bisimplicial presheaves $X_\sigma$ and $X_{\sigma i}$. The map $\sigma\colon T\times\delta_!\Delta[n]\rt N\cat{C}$ is a string of morphisms
$$\xymatrix{
\sigma = \Big( c_0\ar[r]^{\sigma_1} & c_1\ar[r] & \dots \ar[r]^{\sigma_n} & c_n\Big)
 }$$ 
 in the category $\cat{C}(T)_n$. For any object $R$ in the site $\cat{S}$, an element of the set $X_\sigma(R)_{p, q}$ is a quadruple
 $$
 \big(f, \alpha, \beta, x\big)
 $$
 where $f\colon R\rt T$ is a map in $\cat{S}$, $\alpha$ and $\beta$ are maps in $\Del$
 $$\xymatrix{
 \alpha\colon [p]\ar[r] & [n] & \beta\colon [q]\ar[r] & [n]
 }$$
 and $x\in X(R)_q$ is an element whose image under $\pi\colon X\rt \mm{ob}(\cat{C})$ satisfies
 $$
 \pi(x) = \beta^*\big(c_{\alpha(0)}\cdot f\big).
 $$
 An object of $X_{\sigma i}$ is a similar quadruple $(f, \alpha, \beta, x)$ satisfying the additional condition that there is some $l=0, \dots, \hat{k}, \dots, n$ such that $\alpha$ and $\beta$ both miss $l$.
 
 Now consider the bisimplicial presheaves $X_\sigma^0$ and $X_{\sigma i}^0$ whose $(p, q)$-simplices at $R$ are quadruples $(f, \alpha, \beta, x)$ exactly as before, except that we require
 $$
 \pi(x) = \beta^*(c_0\cdot f)\in \mm{ob}(\cat{C})(R)_q
 $$
 (so $c_0$ instead of $c_{\alpha(0)}$). These bisimplicial presheaves fit into a commuting square
 $$\xymatrix{
 X^0_{\sigma i}\ar[d]_{\ol{\sigma}_*}\ar[r] & X^0_\sigma\ar[d]^{\ol{\sigma}_*}\\
 X_{\sigma i}\ar[r] & X_\sigma
 }$$
 where the vertical maps $\ol{\sigma}_*=(\sigma_{\alpha(0)}\circ \dots \circ \sigma_1)_*$ are induced by the action of $\cat{C}$ on $X$. The top inclusion $X^0_{\sigma i}\rt X^0_\sigma$ fits into a pullback diagram of bisimplicial presheaves
 $$\xymatrix{
 X^0_{\sigma i}\ar[d]\ar[r] & X^0_\sigma\ar[d]\ar[r] & X(c_0)\ar[r]\ar[d] & X\ar[d]\\
 T\times\delta_!(\Lambda^k[n])\ar[r] & T\times \delta_!\Delta[n]\ar[r] & T\times \Delta[n]\ar[r] & \mm{ob}(\cat{C})
 }$$
 where all objects in the most right square are constant in one simplicial direction (the $p$-direction, in the above notation). Since the diagonal functor $\delta^*$ preserves limits, it follows that $\delta^*(X^0_{\sigma i})\rt \delta^*(X^0_\sigma)$ is the pullback of a (local) weak equivalence along the (local) fibration $X\rt \mm{ob}(\cat{C})$. Lemma \ref{lem:pullbackalonglocalfib} then implies that $\delta^*(X^0_{\sigma i})\rt \delta^*(X^0_\sigma)$ is a (local) weak equivalence as well.
 
 To finish the proof, it remains to verify that the two vertical maps $\ol{\sigma}_*$ induce $\lambda$-equivalences on the diagonals. But for a fixed $p$, the action map $\ol{\sigma}_*\colon X_\sigma^0\rt X_{\sigma}$ is a coproduct over $\alpha\colon [p]\rt [n]$ of maps of the form
 $$\xymatrix{
 X(c_0)\ar[rr]\ar[rd] & & X(c_{\alpha(0)})\ar[ld]\\
 & T\times \Delta[n].
 }$$
 These maps are all $\lambda$-equivalences of simplicial presheaves by assumption, so the induced map on diagonals is a $\lambda$-equivalence by Lemma \ref{lem:realizationlemma}.
\end{proof}

\section{Examples}\label{sec:ex}
\begin{example}[Quillen's Theorem B]\label{ex:quillenb}
Let $f\colon \cat{D}\rt \cat{C}$ be a functor between categories. Let $X_c=N(f/c)$ be the nerve of the comma category $f/c$ for $c\in\cat{C}$. These $X_c$ form a covariant diagram of simplicial sets indexed by $\cat{C}$. The category $\cat{C}$ acts by weak equivalences on this diagram if for each $\alpha\colon c\rt c'$ in $\cat{C}$, the functor $f/c\rt f/c'$ induces a weak equivalence on nerves.

As a very special case of Theorem \ref{thm:quillenb}, we find that if this is the case, then $X_c$ is the homotopy fibre of
$$\xymatrix{
\hocolim X\ar[r] & B\cat{C}.
}$$
The space $\hocolim X$ is the nerve of the category $f/\cat{C}$ and the spaces $X_c$ are the nerves of the fibres of the functor $f/\cat{C}\rt \cat{C}$ \cite{tho79}.

There is an inclusion $\cat{D}\rt f/\cat{C}$ sending $d$ to $(d, f(d) \stackrel{=}{\rt} f(d))$, which is left adjoint to the obvious projection $f/\cat{C}\rt \cat{D}$. This functor induces a homotopy equivalence on nerves, so that the map $\hocolim X\rt B\cat{C}$ is homotopy equivalent to the map $B\cat{D}\rt B\cat{C}$. We therefore obtain Quillen's original Theorem B, identifying the homotopy fibre of $B\cat{D}\rt B\cat{C}$ over $c\in\cat{C}$ with the nerve of $f/c$.

Theorem \ref{thm:quillenb} gives an extension to localizations (e.g., to the case where each $X_c\rt X_{c'}$ is a homology isomorphism), as well as to functors $\cat{D}\rt \cat{C}$ between (pre)sheaves of categories on a site $(\cat{S}, J)$.
\end{example}
\begin{example}[Homotopy colimits and Puppe's theorem]\label{ex:puppe}
Let ${\mathcal{I}}$ be a small category and let $X$ and $Y$ be two ${\mathcal{I}}$-indexed diagrams of simplicial sets. A natural transformation $f\colon Y\rt X$ is called \emph{homotopy cartesian} if for any morphism $\alpha\colon i\rt j$ in ${\mathcal{I}}$, the naturality square
\begin{equation}\label{diag:cartesiantransf}\vcenter{\xymatrix{
Y_i\ar[r]^{\alpha_*}\ar[d]_{f_i} & Y_j\ar[d]^{f_j}\\
X_i\ar[r]_{\alpha_*} & X_j
}} 
\end{equation}
is a homotopy pullback. Puppe's theorem \cite{pup74} states that for any homotopy cartesian transformation $f$ and any $i_0\in{\mathcal{I}}$, the square
\begin{equation}\label{diag:pullbackoncolim}\vcenter{\xymatrix{
Y_{i_0}\ar[r]\ar[d] & \hocolim Y_i\ar[d]\\
X_{i_0}\ar[r] & \hocolim X_i
 }}\end{equation}
is a homotopy pullback. This theorem is in fact a special case of Theorem \ref{thm:quillenb} (for the trivial site, so for simplicial \emph{sets} rather than simplicial (pre)sheaves). Indeed, let $\cat{C}$ be the simplicial category $X_{{\mathcal{I}}}$ with space of objects $\tilde{X}=\coprod_{i\in{\mathcal{I}}} X_i$ and space of morphisms
$$
\mm{mor}(X_{{\mathcal{I}}}) = \coprod_{i\rt j} X_i.
$$
The natural transformation $f$  defines an action of $X_{{\mathcal{I}}}$ on $\tilde{Y}=\coprod_{i\in {\mathcal{I}}} Y_i$.

After replacing $Y\rt X$ by a fibration in the projective model structure on $\sSet^{{\mathcal{I}}}$, the hypothesis on the squares \eqref{diag:cartesiantransf} mean precisely that $X_{{\mathcal{I}}}$ acts by weak homotopy equivalences. The space $BX_{{\mathcal{I}}}$ is a model for $\hocolim_i X_i$, and Theorem \ref{thm:quillenb} gives for this special case that \eqref{diag:pullbackoncolim} is a homotopy pullback.

Still working on the trivial site, Theorem \ref{thm:quillenb} gives variations of Puppe's theorem for left Bousfield localizations. For example, suppose that all the squares \eqref{diag:cartesiantransf} are ``homology cartesian'', in the sense that for each vertex $x\in X_i$, the map from the homotopy fibre of $f_i$ over $x$ to the one of $f_j$ over $\alpha_*(x)$ is a homology equivalence. Then the map from $Y_{i_0}$ to the homotopy pullback inscribed in \eqref{diag:pullbackoncolim} is also a homology equivalence.

For a left Bousfield localization $\lambda$ of the model category $\mm{sPSh}(\cat{S})_J$ or $\mm{sSh}(\cat{S}, J)$, we obtain a similar result for a map $f\colon Y\rt X$ between ${\mathcal{I}}$-diagrams of simplicial (pre)sheaves: Theorem \ref{thm:quillenb} states that the map
$$\xymatrix{
Y_{i_0}\ar[r] & X_{i_0}\times^h_{\hocolim X_i} \hocolim Y_i
}$$
(homotopy pullback in the non-localized model structure) is a $\lambda$-weak equivalence whenever the map between homotopy fibres
$$\xymatrix{
\mm{hofib}(Y_i)_x\ar[r] & \mm{hofib}(Y_j)_{\alpha_*(x)}
}$$
is a $\lambda$-weak equivalence for each $i\in{\mathcal{I}}$ and each vertex $x\in X_i(S)$. If the localization $\lambda$ is left exact, then Theorem \ref{thm:quillenb} translates into the statement that if each square \eqref{diag:cartesiantransf} is homotopy cartesian in the $\lambda$-localized model structure, then so is each pullback square \eqref{diag:pullbackoncolim}. This is a version of Puppe's theorem for $\infty$-toposes, which is also referred to as \emph{descent}, cf.\ \cite{rez98} or \cite[Chapter 6.1.3]{lur09}. 
\end{example}
\begin{example}[Grouplike monoids]\label{ex:grouplikemonoids}
Let $(\cat{S}, J)$ be a site and $M$ a presheaf of simplicial monoids on $\cat{S}$. Then $M$ acts on itself by left multiplication and we obtain a pullback square
$$\vcenter{\xymatrix{
M\ar[r]\ar[d] & B(M_M)\ar[d]\\
\ast\ar[r] & B(M).
}}$$
The simplicial presheaf $B(M_M)$ is contractible, since the unit element is an initial object of the simplicial category $M_M$. For $S\in\cat{S}$ and $m\in M_0(S)$, left multiplication determines a map $m_*\colon M_{/S}\rt M_{/S}$, where $M_{/S}=S\times M$. If each such $m_*$ is a $\lambda$-equivalence, then it follows from Theorem \ref{thm:quillenb} (and Lemma \ref{lem:actingbyweonvertices}) that
$$\xymatrix{
M\ar[r] & \Omega B M
}$$
is a $\lambda$-equivalence as well. 

There is often a more familiar criterion for the above condition, in terms of the sheaf $\pi_0^\lambda(M)$ associated to the presheaf
$$\xymatrix{
\cat{S}^\op\ar[r] & \Set; \hspace{4pt} S\ar@{|->}[r] & \Hom_{\ho(\mm{sPSh}(\cat{S})_{J, \lambda})}\big(S, M\big).
}$$
To state this criterion, let us assume that for any $\lambda$-equivalence $X\rt X'$ between simplicial presheaves and any $S\in \cat{S}$, the map $X\times S\rt X'\times S$ is a $\lambda$-equivalence. This holds in various cases, e.g.\ for $\infty$-toposes (cf.\ Example \ref{ex:puppe}) and for $\field{A}^1$-model structures \cite{mor99} (cf.\ Example \ref{ex:pinfinity}). It follows that $f\times g\colon X\times Y\rt X'\times Y'$ is a $\lambda$-equivalence if $f$ and $g$ are. In particular, if $M\rt M'$ is an (injectively) fibrant replacement of $M$ in $\mm{sPSh}(\cat{S})_{J, \lambda}$, then $M'$ inherits a multiplication $\mu'$ via
$$\xymatrix{
M\times M\ar[d]_\sim \ar[r]^-\mu & M\ar[d]^\sim\\
M'\times M'\ar@{..>}[r]_-{\mu'} & M'.
}$$
This is unital and associative up to homotopy, so that homotopy classes of maps into $M'$ form a monoid and $\pi_0^\lambda(M)$ is a sheaf of monoids. \emph{The map $M\rt \Omega BM$ is a $\lambda$-equivalence whenever $\pi_0^\lambda(M)$ is a sheaf of groups.}

To see this, take $S\in\cat{S}$ and $m\in M_0(S)$, with image $m'$ in $M'_0(S)$. To see that $m_*\colon M_{/S}\rt M_{/S}$ is a $\lambda$-equivalence, it suffices to verify that $m'_*\colon M'_{/S}\rt M'_{/S}$ is a $\lambda$-equivalence. Because $\pi_0^\lambda(M)$ is a sheaf of groups, there is a cover $\alpha_i\colon S_i\rt S$ such that each $m'_i=\alpha_i^*(m')$ admits a homotopy inverse $n_i\in M'_0(S_i)$. It follows that each
$$\xymatrix{
(m_i')_*\colon M'_{/S_i}\ar[r] & M'_{/S_i}
}$$
is a homotopy equivalence. Similarly, the restriction of $m'$ to an iterated pullback $S_{i_0\dots i_n}$ admits a homotopy inverse and $(m'_{i_0\dots i_n})_*$ is a homotopy equivalence as well. These weak equivalences assemble into a natural weak equivalence of bisimplicial presheaves
$$\xymatrix@R=1.2pc{
\coprod\limits_{i_0\dots i_n} M'_{/S_{i_0\dots i_n}}\ar[rr]^\sim\ar[rd] & & \coprod\limits_{i_0\dots i_n} M'_{/S_{i_0\dots i_n}}\ar[ld]\\
& \coprod\limits_{i_0\dots i_n} S_{i_0\dots i_n}.
}$$
The realization of this natural weak equivalence is weakly equivalent to the map $m'_*\colon M'_{/S}\rt M'_{/S}$ over $S$ (for instance by Puppe's theorem, cf.\ Example \ref{ex:puppe}), so that $m'_*$ and $m_*$ are $\lambda$-equivalences.
\end{example}
\begin{example}[Infinite projective space]\label{ex:pinfinity}
Consider a site $(\cat{S}, J)$ endowed with a sheaf of commutative rings $R$ and let us use $\field{A}^1$ to denote the sheaf of sets underlying $R$. Let $\mm{sPSh}(\cat{S})_{J, \field{A}^1}$ be the left Bousfield localization at all projection maps
$$\xymatrix{
X\times \field{A}^1\ar[r] & X.
}$$
This model category describes `$\field{A}^1$-homotopy theory over $R$'. In particular, two maps $f, g\colon X\rt Y$ describe the same map in the homotopy category of $\mm{sSh}(\cat{S}, J)_{\field{A}^1}$ if there exists an $\field{A}^1$-homotopy
$$\xymatrix{
H\colon X\times \field{A}^1\ar[r] & Y
}$$
such that $H\big|_{X\times \{0\}}=f$ and $H\big|_{X\times \{1\}}=g$.

Let $\field{G}_m\subseteq \field{A}^1$ be the sub-presheaf of invertible elements and let $\field{A}^{n+1}_*$ be the union
$$
\bigcup_{i=0}^{n} \field{A}^i\times \field{G}_m \times \field{A}^{n-i} \subseteq \field{A}^{n+1}.
$$
The presheaf $\field{G}_m$ is a presheaf of groups under multiplication, which acts on $\field{A}^{n+1}_*$ via
$$\xymatrix{
\field{G}_m\times \field{A}^{n+1}_\ast\ar[r] & \field{A}^{n+1}_\ast; \hspace{4pt} z\cdot (x_0, \dots, x_n) = (zx_0, \dots, zx_n).
}$$
This action is free, with quotient $\field{P}^n=\field{A}^{n+1}_*/\field{G}_m$ given by the $n$-th projective space. The projective spaces fit into a sequence
$$\xymatrix@C=4pc{
\dots\ar[r]^{x\mapsto (x, 0)} & \field{A}^{n+1}_\ast\ar[r]^{x\mapsto (x, 0)}\ar[d] & \field{A}^{n+2}_\ast\ar[d]\ar[r]^{x\mapsto (x, 0)} & \dots \ar[r] & \field{A}^\infty_\ast\ar[d]^q\\
\dots\ar[r]_-{[x]\mapsto [x: 0]} & \field{P}^n\ar[r]_-{[x]\mapsto [x:0]} & \field{P}^{n+1}\ar[r]_-{[x]\mapsto [x:0]} & \dots \ar[r] & \field{P}^\infty
}$$
whose colimit $\field{P}^\infty$ is the quotient of the colimit $\field{A}^\infty_*$ by the (free) action of $\field{G}_m$ given by $z\cdot (x_0, \dots, x_n, 0, \dots)= (zx_0,\dots, zx_n, 0, \dots)$.

The presheaf $\field{A}^\infty_\ast$ can be identified with the presheaf of polynomials with coefficients in $R$, with at least one invertible coefficient. Multiplication of polynomials then endows $\field{A}^\infty_\ast$ and its quotient $\field{P}^\infty$ with the structure of a commutative monoid. Let us use the criterion of Example \ref{ex:grouplikemonoids} to verify that
$$\xymatrix{
\field{P}^\infty\ar[r] & \Omega B(\field{P}^\infty)
}$$
is an $\field{A}^1$-weak equivalence. In fact, $\pi_0^{\field{A}^1}(\field{P}^\infty)$ is the terminal sheaf: for any $S\in \field{S}$ and any point $[x]=[x_0:\ldots : x_n:0 :\ldots]\in \field{P}^\infty(S)$, there are $\field{A}^1$-paths
\begin{align*}
[x_t] &= [x_0: x_1: \ldots : x_n : t : 0: \ldots]\\
[y_t] &= [tx_0: tx_1:\ldots : tx_n : 1 : 0:\ldots ]\\
[z_t] &= [t : 0 : \ldots : 0: 1: 0:\ldots]\\
[w_t] &= [1: 0 : \ldots : 0: t:0:\ldots]
\end{align*}
connecting the point $[x]$ to the unit element $[1]=[1: 0: 0:\dots]$ of $\field{P}^\infty$. Identifying $\field{A}^1$-homotopic elements in $\pi_0(\field{P}^{\infty})$ therefore yields the terminal sheaf, which implies that $\pi_0^{\field{A}^1}(\field{P}^\infty)$ is terminal as well (by \cite[Corollary 3.22]{mor99}).
\end{example}
Examples \ref{ex:groupcompletion} and \ref{ex:bglr} and Variant \ref{var:variantofgpcompl} generalize the classical argument of the group completion theorem (see \cite{mcd75,qui71}) to Bousfield localizations of simplicial (pre)sheaves. We only describe the case of simplicial sheaves, the case of simplicial presheaves being completely analogous.
\begin{example}[Group completion]\label{ex:groupcompletion}
Let $(\cat{S}, J)$ be a site and consider the functor
\begin{equation}\label{diag:homologysheaves}\xymatrix{
h_\ast\colon \mm{sSh}(\cat{S}) \ar[r] & \mm{Sh}(\cat{S}; \mm{Ab}_\mm{gr})
}\end{equation}
sending each simplicial sheaf $X$ to its homology sheaves, i.e.\ the associated sheaves of the presheaves $H_*(X(-); \field{Z})$. This functor has the following properties:
\begin{enumerate}[leftmargin=*]
\item It sends local weak equivalences to isomorphisms of sheaves of graded abelian groups.
\item If $X\colon \cat{I}\rt \mm{sSh}(\cat{S})$ is a filtered diagram of simplicial sheaves, then the natural map 
$$\xymatrix{
\colim h_\ast(X_i)\ar[r] & h_\ast(\hocolim X_i)
}$$
is an isomorphism. 
\item Let $X$ and $Y$ be two $\cat{I}$-indexed diagrams of simplicial sheaves and let $f\colon X\rt Y$ be a natural transformation between them. If each $h_*(X_i)\rt h_*(Y_i)$ is an isomorphism, then the map $h_*(\hocolim X)\rt h_*(\hocolim Y)$ is an isomorphism.
\item It is lax symmetric monoidal, i.e.\ there are natural maps 
$$\xymatrix{
h_*(X)\otimes h_*(Y)\ar[r] & h_*(X\times Y) & \field{Z}\rt h_*(\ast)
}$$
where $\otimes$ denotes the usual tensor product of sheaves of graded abelian groups. In particular, $h_*$ sends simplicial monoids to graded rings.
\item $h_*$ is part of an indexed functor in the following sense. For every sheaf (of sets) $S$, its category of elements $\cat{S}/S$ inherits a natural Grothendieck topology from $(\cat{S}, J)$. As in \eqref{diag:homologysheaves}, there is a functor $(h_{*})_{/S}$ taking the homology of simplicial sheaves over $\cat{S}/S$, which satisfies conditions (1) - (4). For any map of sheaves $f\colon S\rt T$, these functors fit into a square which commutes up to natural isomorphism
$$\xymatrix@C=3pc{
\mm{sSh}(\cat{S}/T) \ar[r]^-{(h_*)_{/T}}\ar[d]_{f^*} & \mm{Sh}(\cat{S}/T; \mm{Ab}_\mm{gr})\ar[d]^{f^*}\\
\mm{sSh}(\cat{S}/S) \ar[r]_-{(h_*)_{/S}} & \mm{Sh}(\cat{S}/S; \mm{Ab}_\mm{gr}).
}$$
Here $f^*$ restricts a (simplicial) sheaf along the functor $\cat{S}/S\rt \cat{S}/T$.
\end{enumerate}
Conditions (1) - (3) imply that there exists a left Bousfield localization $\mm{sSh}(\cat{S}, J)_{h_*}$ of the Joyal model structure whose weak equivalences are the $h_*$-isomorphisms (cf.\ the appendix of \cite{bou75}). Condition (5) expresses the local nature of the functor $h_*$; for example, it implies that there is a natural map of sheaves $\pi_0(X)\rt h_0(X)$.

Let $M$ be a sheaf of simplicial monoids on $\cat{S}$ and suppose that $M$ admits a countable set of global sections $m_i\colon \ast\rt M$ such that the map 
$$\xymatrix{
(m_i)_{i\in \field{N}}\colon \field{N}\ar[r] & M
}$$
induces a surjection on $\pi_0$-sheaves. In this case, the \emph{group completion theorem} asserts that the map
$$\xymatrix{
h_*(M)[\pi_0(M)^{-1}]\ar[r] & h_*(\Omega BM)
}$$
is an isomorphism if the sheaf $\pi_0(M)$ is contained in the center of $h_*(M)$.

To see this, let $M_s$ denote the simplicial sheaf obtained as the (homotopy) colimit of the sequence of right multiplication maps
\begin{equation}\label{diag:localizationsequence}\vcenter{\xymatrix@C=3pc{
M\ar[r]^{(-)\cdot m_{i_1}} & M\ar[r]^{(-)\cdot m_{i_2}}\ar[r] & \dots
}}\end{equation}
where each $m_i$ occurs infinitely many times. It follows that
$$
h_*(M_s)\cong \colim\Big(\xymatrix{h_*(M)\ar[r]^{m_{i_1}} & h_*(M)\ar[r]^{m_{i_2}} & h_*(M)\ar[r] & \dots}\Big).
$$
Because $\pi_0(M)$ is contained in the center of $h_*(M)$, this colimit has the structure of an associative algebra. Since every local section of $\pi_0(M)$ agrees with the restriction of some global section $m_i$, we have that
$$
h_*(M_s)\cong h_*(M)[\pi_0(M)^{-1}].
$$ 
It therefore suffices to provide an $h_*$-isomorphism $M_s\rt \Omega BM$. To do this, note that left multiplication turns \eqref{diag:localizationsequence} into a sequence of left $M$-modules, so that $M_s$ is a left $M$-module as well. We obtain a pullback square of simplicial sheaves
$$\vcenter{\xymatrix{
M_s\ar[r]\ar[d] & B((M_s)_M)\ar[d]\\
\ast\ar[r] & BM.
}}$$
The simplicial sheaf $B((M_s)_M)$ is weakly contractible, being a filtered colimit of simplicial sheaves $B(M_M)$ (see Example \ref{ex:grouplikemonoids}). By Theorem \ref{thm:quillenb}, the map $M_s\rt \Omega BM$ is an $h_*$-isomorphism if $M$ acts on $M_s$ by $h_*$-isomorphisms.

To see that $M$ acts on $M_s$ by $h_*$-isomorphisms, we can use (5) to work locally. Given an element $m\colon S\times\Delta[0]\rt M$, we may therefore assume that $m$ is homotopic to one of the global elements $m_i\colon \ast\rt M$, restricted to $S$. Then $m$ acts by $h_*$-isomorphisms as soon as $m_i$ acts by $h_*$-isomorphisms on $M_\infty$. The map
$$\xymatrix@C=3pc{
h_*(M_s)\cong h_*(M)[\pi_0(M)^{-1}]\ar[r]^{m_i\cdot(-)} & h_*(M)[\pi_0(M)^{-1}] \cong h_*(M_s)
}$$
arises from left multiplication by $m_i$ in $h_*(M)$, which becomes an isomorphism on $h_*(M)[\pi_0(M)^{-1}]$ by construction.
\end{example}
\begin{example}\label{ex:bglr}
Suppose that $R$ is a sheaf of commutative rings on $\cat{S}$. For each $n$, let $\mm{GL}_n(R)\subseteq R^{n\times n}$ be the subsheaf of matrices with invertible determinant. Consider the monoid $M=\coprod_n B\mm{GL}_n(R)$ whose multiplication is induced by the block sum of matrices $\mm{GL}_n(R)\times \mm{GL}_m(R)\rt \mm{GL}_{n+m}(R)$. There is an isomorphism of simplicial sheaves
$$
M_s \cong \field{Z}\times B\mm{GL}_\infty(R) := \field{Z}\times \colim\Big(\hspace{-2pt}\xymatrix{B\mm{GL}_1(R)\ar[r]^-{(-)\oplus 1} & B\mm{GL}_2(R)\ar[r]^-{(-)\oplus 1} & \dots}\Big)
$$
because $\pi_0(M)\cong\field{N}$ is generated by a single element $1$. The group completion theorem of Example \ref{ex:groupcompletion} now asserts that the map
$$\xymatrix{
\field{Z}\times B\mm{GL}_\infty(R) \ar[r] & \Omega B(M)
}$$
induces an isomorphism on homology sheaves.
\end{example}
\begin{variant}\label{var:variantofgpcompl}
The same argument applies when the (integral) homology functor $h_*$ is replaced by any other functor 
$$\xymatrix{
E_*\colon \mm{sSh}(\cat{S})\ar[r] & \mm{Sh}\big(\cat{S}; \mm{Mod}^\mm{gr}_{E_*(*)}\big)
}$$
which takes values in sheaves of graded modules over a sheaf of graded-commutative rings $E_*(*)$ and satisfies conditions (1) - (5) above.
\end{variant}

\bibliographystyle{abbrv}
\bibliography{bibliography_quillen_b}

\begin{thebibliography}{10}

\bibitem{bou75}
A.~K. Bousfield.
\newblock The localization of spaces with respect to homology.
\newblock {\em Topology}, 14:133--150, 1975.

\bibitem{bro73}
K.~S. Brown.
\newblock Abstract homotopy theory and generalized sheaf cohomology.
\newblock {\em Trans. Amer. Math. Soc.}, 186:419--458, 1973.

\bibitem{sga4}
P.~Deligne.
\newblock {\em Appendix to Th\'eorie des topos et cohomologie \'etale des
  sch\'emas II, Exp. VI.}
\newblock Lecture Notes in Mathematics, Vol. 270. Springer-Verlag, Berlin-New
  York, 1972.
\newblock S\'eminaire de G\'eom\'etrie Alg\'ebrique du Bois-Marie 1963--1964
  (SGA 4), Dirig\'e par M. Artin, A. Grothendieck et J. L. Verdier. Avec la
  collaboration de N. Bourbaki, P. Deligne et B. Saint-Donat.

\bibitem{har80}
B.~Harris.
\newblock Bott periodicity via simplicial spaces.
\newblock {\em J. Algebra}, 62(2):450--454, 1980.

\bibitem{jar87}
J.~F. Jardine.
\newblock Simplicial presheaves.
\newblock {\em J. Pure Appl. Algebra}, 47(1):35--87, 1987.

\bibitem{joy84}
A.~Joyal.
\newblock Letter to {A}lexander {G}rothendieck.
\newblock 1984.

\bibitem{lur09}
J.~Lurie.
\newblock {\em Higher topos theory}, volume 170 of {\em Annals of Mathematics
  Studies}.
\newblock Princeton University Press, Princeton, NJ, 2009.

\bibitem{mak77}
M.~Makkai and G.~E. Reyes.
\newblock {\em First order categorical logic}.
\newblock Lecture Notes in Mathematics, Vol. 611. Springer-Verlag, Berlin-New
  York, 1977.

\bibitem{mcd75}
D.~McDuff and G.~Segal.
\newblock Homology fibrations and the ``group-completion'' theorem.
\newblock {\em Invent. Math.}, 31(3):279--284, 1975/76.

\bibitem{moe89}
I.~Moerdijk.
\newblock Bisimplicial sets and the group-completion theorem.
\newblock In {\em Algebraic {$K$}-theory: connections with geometry and
  topology ({L}ake {L}ouise, {AB}, 1987)}, volume 279 of {\em NATO Adv. Sci.
  Inst. Ser. C Math. Phys. Sci.}, pages 225--240. Kluwer Acad. Publ.,
  Dordrecht, 1989.

\bibitem{mor99}
F.~Morel and V.~Voevodsky.
\newblock {${\bf A}^1$}-homotopy theory of schemes.
\newblock {\em Inst. Hautes \'Etudes Sci. Publ. Math.}, (90):45--143 (2001),
  1999.

\bibitem{pup74}
V.~Puppe.
\newblock A remark on ``homotopy fibrations''.
\newblock {\em Manuscripta Math.}, 12:113--120, 1974.

\bibitem{qui73}
D.~Quillen.
\newblock Higher algebraic {$K$}-theory. {I}.
\newblock pages 85--147. Lecture Notes in Math., Vol. 341, 1973.

\bibitem{qui71}
D.~Quillen.
\newblock On the group completion of a simplicial monoid, {A}ppendix {Q} to
  `{F}iltrations on the homology of algebraic varieties' by {E}. {F}riedlander
  and {B}. {M}azur.
\newblock {\em Mem. Amer. Math. Soc.}, 110(529):x+110, 1994.
\newblock (MIT preprint 1971).

\bibitem{rez98}
C.~Rezk.
\newblock Fibrations and homotopy colimits of simplicial sheaves.
\newblock {\em
  \href{https://arxiv.org/abs/math/9811038}{\textup{arXiv:math/9811038}}},
  1998.

\bibitem{tho79}
R.~W. Thomason.
\newblock Homotopy colimits in the category of small categories.
\newblock {\em Math. Proc. Cambridge Philos. Soc.}, 85(1):91--109, 1979.

\end{thebibliography}

\end{document}